\documentclass [12pt,a4paper]{amsart}

\usepackage{graphicx}
\usepackage[utf8]{inputenc} 
\usepackage{amssymb,amsmath,amsfonts,amsthm,mathtools}
\usepackage{enumitem}
\renewcommand{\raggedright}{\justifying}							
\usepackage{helvet}																		
\newtheorem{theorem}{Theorem}[section]
\newtheorem{lemma}[theorem]{Lemma}
\newtheorem{proposition}[theorem]{Proposition}
\newtheorem{corollary}[theorem]{Corollary}

\theoremstyle{definition}
\newtheorem{definition}{Definition}

\theoremstyle{remark}
\newtheorem{remark}{Remark}
\newtheorem*{notation*}{Notation}
\usepackage[color=red!40]{todonotes}
\newcommand{\Hyp}{\mathbb{H}^3}
\newcommand{\kd}{k_{d}}
\newcommand{\N}{\mathrm{N}}
\newcommand{\SL}{\mathrm{SL}}
\newcommand{\PSL}{\mathrm{PSL}}
\newcommand{\sys}{\mathrm{sys}}
\newcommand{\tr}{\mathrm{tr}}

\newcommand{\K}{\mathrm{Kiss}}
\newcommand{\Od}{\mathcal{O}_d}
\newcommand{\Bd}{\mathrm{SL}_2(\Od)}
\newcommand{\vol}{\mathrm{vol}}
\newcommand{\Isom}{\mathrm{Isom}}

\usepackage{hyperref}

\usepackage{enumitem}

\begin{document}

\title[Hyperbolic 3-manifolds with large kissing number]
        {Hyperbolic 3-manifolds with large kissing number}
\author{Cayo D\'oria}
\author{Plinio G. P. Murillo}
\thanks{D\'oria is grateful for the support of FAPESP grant 2018/15750-9.\\
\indent Murillo was supported by the KIAS Individual Grant MG072601.}

	\begin{abstract}
In this article we construct a sequence $\{M_i\}$ of non compact finite volume hyperbolic $3$-manifolds whose kissing number grows at least as $\vol(M_i)^{\frac{31}{27}-\epsilon}$ for any $\epsilon>0$. This extends a previous result due to Schmutz in dimension $2$.
    \end{abstract}

\maketitle 

\section{Introduction}

Let $M$ be a hyperbolic manifold of finite volume. The \emph{systole} of $M$ is the shortest length of a nontrivial closed geodesic, and it is denoted by  $\sys(M)$. A related invariant is the \emph{kissing number} $\K(M)$, defined as the number of free homotopy classes of oriented closed geodesics in $M$ of length $\sys(M)$.  The study of these invariants has a long history, starting from the moduli space of flat tori (see \cite{S94}).

For any $n \ge 2$, there is no a universal upper bound for the kissing number of hyperbolic $n$-manifolds of finite volume. However, the problem of understanding $\K(M)$ becomes interesting if we fix an upper bound for $\vol(M)$ (the volume of $M$), and we study the asymptotic behavior of $\K(M)$ as $\vol(M)$ goes to infinity. More precisely, for each $n$ we consider the maps 
\[ K_n(v)= \sup \{ \K(M)\mid  M \mbox{ is a hyperbolic } n \mbox{-manifold of } \vol(M) \leq v \}. \]


Before explaining the main contribution of this paper, let us briefly mention what we can say about the asymptotic behaviour of $K_n(v)$. In the search of a hyperbolic manifold $M$ with large $\K(M)$, we could consider such a manifold having a large isometry group $G(M)$, and such that no isometry fix a shortest closed geodesic. On the other hand, by Kazhdan-Margulis Theorem,  it happens that $|G(M)| \leq C_n \vol(M)$, where $C_n$ depends only on $n$. Therefore, we could only expect that \[\limsup_{v \to \infty} \frac{\log(K_n(v))}{\log(v)} \geq 1.\]

When we restrict the study of $K_n(v)$ to the class of arithmetic hyperbolic manifolds, more can be said. In \cite{S94}, Schmutz started the investigation of extremal values of systole and kissing number of hyperbolic surfaces with a fixed area (see also  \cite{MR1408860} and \cite{MR1394752} for earlier contributions in this problem). In \cite{S97}, Schmutz showed  that the sequence of principal congruence subgroups $\Gamma(N)$ of the modular group $\rm{PSL}(2,\mathbb{Z})$ produce hyperbolic surfaces $S(N):=\Gamma(N) \backslash \mathbb{H}^2$ of finite area satisfying
\begin{align}\label{040220.1}
 \K(S(N)) \geq c_1 \mathrm{area}( S(N))^{\frac{4}{3}-\varepsilon}, \hspace{3mm} N\rightarrow\infty, 
\end{align}
for any $\varepsilon >0$ and a universal constant $c_1>0$.  
In the same paper, the author also constructed compact arithmetic hyperbolic surfaces with kissing number having the same growth as in \eqref{040220.1}. On the other hand, more recent works by Parlier \cite{Par13} and Fanoni-Parlier \cite{FP15} prove the upper bound $$K_2(v) \leq A \dfrac{v^2}{\log(1+v)},$$ for some absolute constant $A>0$. Therefore, we have that 

\begin{equation}\label{eq:SchFanPar}
\frac{4}{3} \leq  \limsup_{v \to \infty} \frac{\log(K_2(v))}{\log(v)}  \leq 2.
\end{equation}
The right-hand side of \eqref{eq:SchFanPar} was partially generalized to higher dimensions very recently by Bourque and Petri in \cite{BP19.1}. Indeed, if we define 
{\small \[ K_{n}^{c}(v)= \sup \{ \K(M)\mid  M \mbox{ is a \emph{closed} hyperbolic } n \mbox{-manifold of } \vol(M) \leq v \}, \]}
it follows from Corollary $2$ in [op.cit] that for all $n \ge 2$ it holds
\[\limsup_{v \to \infty} \frac{\log(K^c_n(v))}{\log(v)} \leq 2. \]



Motivated by the results above, this article is the first step toward an understanding of the quantity $\limsup_{v \to \infty} \frac{\log(K_n(v))}{\log(v)}$ in higher dimensions, starting by generalizing \eqref{040220.1} to dimension 3, and then providing the first nontrivial lower bound for $\limsup_{v \to \infty}\frac{\log(K_3(v))}{\log(v)}.$ More precisely, we prove the following


\begin{theorem}\label{thm:main_theorem}
There exist a sequence $\{M_i\}$ of non-compact finite volume hyperbolic 3-manifolds with $\vol(M_i) \to \infty,$ such that  $$\K(M_i)\gtrsim c \frac{\vol(M_i)^{\frac{31}{27}}}{\log(\vol(M_i))},$$
for some constant $c>0$ independent of $M_i$. In particular, $$\limsup_{v \to \infty} \frac{\log(K_3(v))}{\log(v)}\geq \dfrac{31}{27}.$$
\end{theorem}




%

\begin{remark}
It follows from our construction that the manifolds $\{M_i\}$ in Theorem \eqref{thm:main_theorem} have large systole (see Corollary \eqref{cor:final}). 
\end{remark}


Although the proof of Theorem \ref{thm:main_theorem} follows the idea of Schmutz in \cite{S97}, it is not easy to extend the technique to dimensions greater than 2. The main difficulty consists in constructing discrete groups $\overline{\Gamma}_i<\PSL_2(\mathbb{C})$ with finite covolume such that, first we have a control on which elements in $\overline{\Gamma}_i$ realize the systole of $M_i=\overline{\Gamma}_i\backslash \mathbb{H}^3,$ and second we can guarantee the existence of a large number of conjugacy classes of loxodromic elements in $\overline{\Gamma}_i$ realizing the systole. We will follow the equivalent approach of finding subgroups $\Gamma_i$ of $\SL_2(\mathbb{C})$ whose projections $\overline{\Gamma}_i$ into $\PSL_2(\mathbb{C})$ satisfy the properties above.  

An important ingredient in \cite{S97} is the high multiplicity of nonconjugated elements in $\SL_2(\mathbb{Z})$ of a given trace, a result proven by Siegel. As it was mentioned before, the surfaces come from principal congruence subgroups of $\SL_2(\mathbb{Z}),$ and the level of the group depends on the given trace. The result then follows from the fact that the translation length in $\SL_2(\mathbb{R})$ is determined by the trace. In dimension 3 this is no longer the case, and a new version of a trace-length relation in $\SL_2(\mathbb{C})$ is needed. This is presented in Section \ref{sec:hypgeo} after we recall some special features of the geometry of the hyperbolic $3$-space (see Proposition \ref{prop:displacement_equality}).

There is no generalization of Siegel's result for discrete subgroups of $\SL_2(\mathbb{C})$. Instead, we are able to apply a theorem due to Sarnak (\cite{Sar83}) on averages of class numbers of binary quadratic forms over the ring of integers of imaginary quadratic fields (see Theorem \ref{thm:Sarnak_growth}). 
It is worth to mention that the relation given in Theorem \ref{thm:Sarnak_growth} is clear only for Bianchi groups. Hence, our approach cannot be used for compact manifolds.
In Section \ref{sec:quadforms} we review the definitions and important facts about these quadratic forms, and the relation with discrete subgroups of isometries of $\Hyp$. 
This allows us to construct the groups $\Gamma_i$ as suitable index-two normal extension of principal congruence subgroups of $\SL_2(\Od),$ where $\Od$ denotes the ring of integers of an imaginary quadratic field with class number one. This is the content of Section \ref{sec:cong_subpgs_and_systoles}. The proof of Theorem \ref{thm:main_theorem} is completed in Section \ref{sec:proofmainthm}.

\textit{Acknowledgements.} We would like to express our gratitude to the referee, whose careful reading and professional feedback has undoubtedly allowed us to share a much improved article with the community. We also want to thank B. Petri and M. F. Bourque for correspondence. The first author thanks the KIAS for hospitality and financial support during a two-weeks research visit in Korea. 

\vspace{-0.1cm}

\section{Hyperbolic geometry}\label{sec:hypgeo}
\subsection{Hyperbolic 3-manifolds}
The set $\mathbb{H}^3=\mathbb{C}\times (0,\infty)$ equipped with the line element $ds^2=(dx^2+dy^2+dt^2)/t^2$ is a model of the unique simply connected Riemannian 3-manifold with constant sectional curvature $-1,$ called the \emph{hyperbolic 3-space}. The group $\SL_2(\mathbb{C})$ acts on the complex plane by linear fractional transformations: For $M=\begin{pmatrix}
a & b \\
c & d
\end{pmatrix}\in\SL_2(\mathbb{C}),$
the action is defined by $$M\cdot z=\frac{az+b}{cz+d}.$$

 This action extends uniquely to an action in $\Hyp$ via the Poincar\'e extension, and the group $\PSL_2(\mathbb{C})=\SL_2(\mathbb{C})/\{\pm I\}$ identifies with the orientation preserving isometry group $\Isom^{+}(\Hyp)$ of $\Hyp$. By a hyperbolic 3-manifold we mean a quotient $M=\Gamma\backslash\Hyp,$ where $\Gamma$ is a discrete torsion-free subgroup of $\PSL_2(\mathbb{C})$. In this article we will deal with discrete subgroups in $\SL_2(\mathbb{C}),$ and then project them to $\PSL_2(\mathbb{C})$. 

\subsection{The complex translation length in $\SL_2(\mathbb{C})$} 
The elements in $\SL_2(\mathbb{C})$ are classified accordingly to their trace. More precisely, $B\in\SL_2(\mathbb{C})$, $B\neq \rm{Id}$ is said to be \emph{elliptic} if $\tr(B)\in\mathbb{R}$ and $|\tr(B)|<2,$ \emph{parabolic} if $\tr(B)=\pm2$ and \emph{loxodromic} otherwise. For
any loxodromic element $B\in\SL_2(\mathbb{C})$ there exists a geodesic $\beta$ in $\Hyp$ (the axis of $B$) such that $B$ is a screw motion translating along $\beta$ a distance $\ell(B),$ and simultaneously rotating about it an angle $\theta(B)$. The pair of real numbers $\ell(B)$ and $\theta(B)$ are the \emph{translation length} and the \emph{rotational angle} of $B$ respectively. The complex number $\ell(B)+i\theta(B)$ is called the \emph{complex translation length of $B$}. It is known that $\pm \tr(B)$ determines the complex translation length. More precisely the following relation holds (\emph{c.f} \cite[Lem. 12.1.2]{MR03})

\begin{equation}
\label{eq:trace_length_1}
\cosh\left(\frac{\ell(B)+i\theta(B)}{2}\right)=\frac{\pm \tr(B)}{2}.
\end{equation}

\subsection{A length-trace relation}

The free homotopy classes of oriented closed geodesics in a hyperbolic manifold $M=\Gamma\backslash\Hyp$ are in one-to-one correspondence with the conjugacy classes of loxodromic elements in $\Gamma,$ and the length of the geodesic corresponds to the translation length of the loxodromic element. Therefore, in order to understand the length of closed geodesics in $M$ we need to study the translation length of the loxodromic elements in $\Gamma$. For the purpose of this article, we need to determine $\ell(B)$ from $B,$ and not only the complex translation length. The following result provides this relation (\textit{c.f} \cite[Lemma 5.1]{G15}).

\begin{proposition}\label{prop:displacement_equality}
For any loxodromic element $B\in \SL_2(\mathbb{C})$ we have
$$4\cosh(\ell(B))=|\tr(B)^2|+|\tr(B)^2-4|.$$

\end{proposition}

\begin{proof}

Let $X\in\SL_2(\mathbb{C})$ be any loxodromic element.
Writing $\tr(X)=x+iy,$ by \eqref{eq:trace_length_1} we have that $x=\pm 2\cosh\left(\frac{\ell(X)}{2}\right)\cos\left(\frac{\theta(X)}{2}\right)$ and $y= \pm 2\sinh\left(\frac{\ell(X)}{2}\right)\sin\left(\frac{\theta(X)}{2}\right)$.
Therefore 

%

%
%
\[\frac{x^2}{\left(2\cosh\left(\frac{\ell(X)}{2}\right)\right)^2}+\frac{y^2}{\left(2\sinh\left(\frac{\ell(X)}{2}\right)\right)^2}=1.\]

It means that $\tr(X)$ lies in the ellipse centered in the origin, with focal points $\pm 2,$ and that intersects the real axis in the points $\pm 2\cosh\left(\frac{\ell(X)}{2}\right)$. Hence $|\tr(X)-2|+|\tr(X)+2|=4\cosh\left(\frac{\ell(X)}{2}\right).$ 
Taking $X=B^2,$ and using the fact that    $\tr(B^2)=\tr(B)^2-2$ and $\ell(B^2)=2\ell(B)$ we obtain the result.\qedhere
\end{proof}

\section{Quadratic forms and loxodromic elements}\label{sec:quadforms}

Let $d>0$ be a square free rational integer. Consider the quadratic number field $\kd=\mathbb{Q}(\sqrt{-d}),$ and let $\Od$ denote the ring of integers of $\kd$. From now on, we will consider only $\kd$  with class number one, i.e. $d\in\lbrace 1,2,3,7,11,19,43,67,163\rbrace$ \cite{Gol85}.
It turns out that loxodromic elements in $\Bd$ are in correspondence with  certain binary quadratic forms with coefficients in $\Od$. The aim of this section is to recall this correspondence following mainly the exposition by Sarnak in \cite{Sar83}.

\subsection{A Diophantine equation} The study of binary quadratic forms over $\Od$ is closely related to the study of the solutions $t, u\in\Od$ to the Pell's type equation 
\begin{equation}\label{eq:Pell_Type}
t^2-u^2D=4.
\end{equation}

Consider the set
\[\mathfrak{D}=\{ D \in \Od \mid D \equiv x^2 (\hspace{-4mm}\mod 4) \mbox{ and } D \mbox{ is not a perfect square} \}. \]


For any $D\in \mathfrak{D}$, let $K=k_d(\sqrt{D})$ (the branch of the square root is chosen such that the argument of $\sqrt{D}$ lies in $[0,\pi)$), and let $\mathcal{O}_{K}$ denotes the ring of integers of $K$. For any solution $(t,u)$ of \eqref{eq:Pell_Type} we associate the complex number
 $$\epsilon_{t,u}=\frac{1}{2}(t+u\sqrt{D}).$$ 
  
  It is clear that $\epsilon_{t,u}\in\mathcal{O}_{K}^{\times},$ and the association $(t,u)\mapsto\epsilon_{t,u}$ provides a group structure on the set of solutions of \eqref{eq:Pell_Type} induced by that on $\mathcal{O}_K^{\times}$.

\begin{definition}
Let $D\in \mathfrak{D}$. A solution $(t_0,u_0)$ of the equation $t^2-u^2D=4$ is called a \emph{fundamental solution} if $|\epsilon_{t_0,u_0}|$ is the smallest possible value larger than one. We write $\epsilon_{t_0,u_0}=\epsilon_D$ in this case.

\end{definition}

Since $|\epsilon_D|>1,$ this definition is equivalent to that in \cite[Pag. 275]{Sar83}. It is also known that when $|D|>4,$ the numbers $\pm\epsilon_D^{\pm n}$ with $n>0$, $n\in\mathbb{Z}$ produce all the solutions to \eqref{eq:Pell_Type} (see \cite[\textit{loc. cit.}]{Sar83}). The following gives us information about the growth of $|\epsilon_D^{n+1}|$.

\begin{proposition} \label{Pliniobreak}
Let $D\in \mathfrak{D}$, and let $\epsilon_D$ be a fundamental solution of $t^2-u^2D=4$. Then, if $\epsilon_D^{n+1}=\frac{1}{2}(t_n+u_n\sqrt{D})$ we have $$|t_n|^2-3<|\epsilon_D|^{2(n+1)}<|t_n|^2+3,$$ for any $n\geq 0$.

\end{proposition}
\begin{proof}

For simplicity, we will write $\epsilon=\frac{1}{2}(t_n+u_n\sqrt{D}),$ and consider the polynomial $f(x)=x^2-t_nx+1$ defined  over $\Od$. The roots of $f$ are precisely $\epsilon$ and $\epsilon^{-1}$. Consider now the polynomial $g(x)=f(x)\overline{f(x)},$ which has roots $\epsilon,\epsilon^{- 1}, \overline{\epsilon}, \overline{\epsilon}^{- 1}$. Factorizing $g$ into linear factors, and comparing the term $x^2$ we obtain that  

$$1+|\epsilon|^2+\epsilon\overline{\epsilon}^{-1}+\epsilon^{-1}\overline{\epsilon}+|\epsilon|^{-2}+1=|t_n|^2+2.$$

On the other hand $|\epsilon|>1,$ so that  by the triangle inequality $|\epsilon\overline{\epsilon}^{-1}+\epsilon^{-1}\overline{\epsilon}+|\epsilon|^{-2}|< 3$. Therefore $-3< |\epsilon|^2-|t_n|^2< 3$ as we claimed.
\end{proof}

\subsection{Quadratic forms}
Let $Q(x,y)=ax^2+bxy +cy^2$ be a binary quadratic form, with $a,b,c\in\Od$. We say that $Q$ is \emph{primitive} if the ideal generated by $\{a,b,c\}$ is $\Od$. Similar to the classical theory of binary quadratic forms over $\mathbb{Z},$ the group $\Bd$ acts on the set of binary quadratic forms over $\Od$ by linear substitution, and we say that $Q$ and $Q'$ are equivalent if they belong to the same orbit under this action. The main invariant of $Q$ is its \textit{discriminant}, which is given by $D=b^2-4ac$. 

Let $B \in \Bd$ be a loxodromic element. We say that $B$ 
is \emph{primitive} if, whenever $B=A^m$ for a positive integer $m$, and $A \in \Bd,$ we have $m=1$ and $B=A$. 
In the case where the quadratic form $Q(x,y)=ax^2+bxy +cy^2$ has discriminant $D\in\mathfrak{D}$, it is possible to associate a primitive loxodromic element in $\Bd$ given by 
\begin{equation}\label{eq:matrixepsilon}
\epsilon_Q=\begin{pmatrix}
\frac{t_0-bu_0}{2} & -cu_0\\
au_0 & \frac{t_0+bu_0}{2}\\
\end{pmatrix},
\end{equation}
where $(t_0,u_0)$ is a fundamental solution to $t^2-Du^2=4$. It turns out that all the primitive loxodromic elements in $\Bd$ can be obtained in this way (see \cite[Thm. 4.1]{Sar83}).



\begin{proposition}[\textbf{Sarnak}]\label{prop:forms_and_loxodromic}
The association $Q\mapsto\epsilon_Q$ is a one-to-one correspondence between the equivalence classes of primitive quadratic forms $Q$ with discriminant in $\mathfrak{D}$, and conjugacy classes of primitive loxodromic elements in $\Bd$.
\end{proposition}

\begin{remark}
For any $\alpha,\beta \in \mathcal{O}_d$, the number $D=\beta^2+4\alpha$ is the discriminant of the  binary primitive quadratic form $Q(x,y)=\alpha x^2+\beta xy- y^2.$ In particular,  any $D \in \mathfrak{D}$ corresponds to a nonempty set of equivalence classes of forms with discriminant $D$. On the other hand, if we take $\alpha=0$, then $D \notin \mathfrak{D}$, i.e. not all discriminants of primitive forms are contained in $\mathfrak{D}$.
\end{remark} 
We define the class number $h(D)$ as the number of equivalence classes of binary primitive quadratic forms with coefficients in $\Od$, and discriminant $D$. For any $x>0$ consider the finite set $\mathfrak{D}_x=\{D\in\mathfrak{D};|\epsilon_D|\leq x\}$. The main result in \cite{Sar83} is to provide an asymptotic behavior for the average of class numbers, when ordered by the size of $\epsilon_D$. More precisely (see \cite[Thm. 7.2]{Sar83}).

\begin{theorem}[\textbf{Sarnak}]\label{thm:Sarnak_growth}
There exists a constant $c_d>0$ depending only on $k_d$ such that 
$$\frac{1}{|\mathfrak{D}_x|}\sum_{D\in\mathfrak{D}_x
}h(D)=\frac{\emph{Li}(x^4)}{c_d x^2}+O(x^\gamma),$$
for any $\gamma>4/3$ as $x\rightarrow\infty,$ and $\emph{Li}(u)=\int_2^{u}\frac{1}{\log t}dt$.
\end{theorem}

\section{Congruence subgroups and shortest closed geodesics}\label{sec:cong_subpgs_and_systoles}

The aim of this section is to construct a class of subgroups of $\Bd$, such that, it is possible to establish which loxodromic elements have shortest translation length. This content is the most technical part of the article, and contains results needed to prove the main theorem in the next section.
 
For any ideal $I\subset\Od$ consider the projection map
$$\pi_I:\SL_2(\Od)\rightarrow \SL_2(\Od/I)$$
given by reduction modulo $I$. The \emph{principal congruence subgroup of level I} is given by $\Bd[I]=\ker(\pi_I).$ In general, a group $\Gamma < \Bd$ is called a \emph{congruence subgroup} if $\Gamma$ contains $\Bd[I]$ for some ideal $I$. The following lemma will be important later on. 

\begin{lemma}\label{Marfloflemma}
If $ T \in \Bd[I],$ then $\tr(T) \equiv 2~ (\mathrm{mod}  ~ I^2).$ 
\end{lemma}

\begin{proof}
We can write $$T=\begin{pmatrix}
1+\beta_1 & \beta_2 \\ \beta_3 & 1+\beta_4
\end{pmatrix},$$ with $\beta_i \in I$ for all $i=1,2,3,4.$ The result follows from the equation $\det(T)=1$ modulo $I^2$.
\end{proof}

Let $D\in\mathfrak{D}$. We can write $D=\beta^2+4\xi$ for some $\xi \in \Od$. For any nontrivial solution $(t,u)$ of $t^2-Du^2=4$ with $t,u \in \Od$ we consider the algebraic numbers $\frac{t-\beta u}{2}=\tau_1$ and $\frac{t+\beta u}{2}=\tau_2$. Note that $\tau_1$ and $\tau_2$ are the roots of the quadratic polynomial $$X^2-tX+(1+\xi u^2) \in \Od[X].$$ Hence, $\tau_1, \tau_2 \in \Od$. Moreover, $\tau_1-\tau_2 \equiv 0~ (\mathrm{mod} ~ u)$. If we define $\tau\in\Od/u\Od$ being the class of $\tau_1$ (equivalently $\tau _2$) modulo $u\Od$, then
$\tau^2 = 1 \mbox{ in } \Od /u \Od,$ since $\tau_1 \tau_2 = 1+\xi u^2$. Let $\mathrm{Id}$ denote the identity matrix in $\SL_2(\Od/u\Od)$. We obtain that the set $\left\{\mathrm{Id},\tau \mathrm{Id} \right\} $ is a normal subgroup of $\SL_2(\Od/u\Od)$. This allow us to define the subgroup 
 $$\Bd_\tau[u]:=\pi_{u\Od}^{-1}\left( \left\{\mathrm{Id},\tau \mathrm{Id} \right\} \right) \lhd \Bd.$$
Hence, $\Bd_\tau[u]$  contains $\Bd[u\Od]$ as a subgroup of index 2, and in particular it is a congruence subgroup.

\begin{remark}\label{rem:tau}
Note that $\tau$ depends only on $t,u$ and $D$. Indeed, if $D=\beta_1^2+4\xi_1=\beta_2^2+4\xi_2,$ then $4$ divides $\beta_1^2-\beta_2^2$. By \cite[Lem. 4.5]{Sar83} $2$ must divide $\beta_1+\beta_2$. Hence,  $ \frac{t+\beta_1 u}{2} - \frac{t-\beta_2 u}{2} \equiv 0~ (\hspace{-3mm}\mod u).$
\end{remark}

Under a suitable condition on $t$ and $u,$ we can characterize the loxodromic elements in $\Bd_\tau[u]$ having the shortest translation length. Before that, we will need the following elementary but useful claim.

\begin{lemma}\label{lem:elementary_lemma}
Let $z,w \in \mathbb{C}$ satisfying  $|z|+|w| \ge 8$ and   $|z| \ge |w|+1 $. Then $|z^2-4| \ge |w^2-4|$.

\end{lemma}

\begin{proof}
The conditions imply that $|z|^2-4\geq|w|^2+4$. The result is then a consequence of the triangle inequality.
\end{proof}

\begin{proposition}\label{systolestimate}
Suppose that $4<|t|<\frac{4}{9}|u|^2$. Then any loxodromic element $B\in\Bd_\tau[u]$ satisfies $$4\cosh(\ell(B))\geq |t|^2+|t^2-4|.$$

In particular, if $\tr(B)=t$ then $B$ determines a closed geodesic in $\Bd_\tau[u]\backslash\Hyp$ of shortest length.
\end{proposition}
\begin{proof}
Consider a primitive quadratic form $Q=ax^2+bxy+cy^2$ over $\Od$  with discriminant $D=b^2-4ac$. Since $\Bd_\tau[u]$ has index $2$ in $\Bd[u\Od]$ we have $$\Bd_\tau[u]=\Bd[u\Od] \cup T \cdot \Bd[u\Od]$$ for any $T \in \Bd_\tau[u] \setminus \Bd[u\Od].$ Let $B\in\Bd_\tau[u]$ be a non-parabolic element. Firstly, suppose that $B \in \Bd[u\Od]$. By Lemma \ref{Marfloflemma} $\tr(B)=2+\xi u^2$ for some $\xi \neq 0$. We claim that in this case
$$4\cosh(\ell(B)) > |t|^2+|t^2-4|.$$ 
Indeed, by hypothesis
\begin{align}\label{incongruence2}
|\tr(B)|  \ge |u|^2-2 > 2|t|-2 >|t|+1,
\end{align}
and $|\tr(B)|+|t| \geq 2|t| >8.$
Therefore, we can apply Proposition \ref{prop:displacement_equality} and Lemma \ref{lem:elementary_lemma}. On the other hand, we have already seen that the matrix $$A=\begin{pmatrix}
\frac{t-bu}{2} & -cu \\ au & \frac{t+bu}{2}
\end{pmatrix}$$
lies in $\Bd$.
By the definition of $\tau,$ we have that $A \in \Bd_\tau[u]$ and $A \notin \Bd[u\Od]$ by \eqref{incongruence2}, since $\tr(A)=t$. Hence, we only need to estimate the displacement of any element of the form $AB$ with $B \in \Bd[u\Od],$ $B \neq \rm{Id}$. We can write \\ $$B=\begin{pmatrix}
1+u\theta_1 & u\theta_2 \\ u \theta_3 & 1+u\theta_4
\end{pmatrix}$$ with $\theta_i \in \Od$ for all $i=1,2,3,4.$ Therefore, modulo $u^2\Od$ we have
$$\tr(AB) \equiv t + u(\theta_1+\theta_4)\frac{t-bu}{2} \equiv t ~ (\hspace{-4mm}\mod u^2),$$ 
since \textbf{$u(\theta_1+\theta_4) \equiv 0 ~(\mathrm{mod} ~ u^2)$} by Lemma \ref{Marfloflemma}. We write $\tr(AB)=t+u^2\eta$ for some $\eta \in \Od$. If $\eta=0$, then $\cosh(\ell(AB))=|t|^2+|t^2-4|$ by Proposition \ref{prop:displacement_equality}. If $\eta \neq 0$, then $|\eta|\geq 1$ as $\eta$ is a quadratic imaginary algebraic integer, and
\begin{align} \label{030220.5}
 |\tr(AB)| \geq |u^2||\eta|-|t|> \frac{9}{4}|t|-|t| > |t|+1 \mbox{ for } |t|>4,
\end{align} 
and we can apply again Lemma \ref{lem:elementary_lemma} in order to conclude that 
$$4\cosh(\ell(AB))=|\tr(AB)|^2+|\tr(AB)^2-4| > |t|^2+|t^2-4|.\qedhere$$
\end{proof}
Let $t,u,D$ be quadratic imaginary algebraic integers satisfying the Pell's type equation \eqref{eq:Pell_Type}, with $D\in\mathfrak{D}$. We provide now sufficient conditions for $t,u$ to satisfy the hypothesis of Proposition \ref{systolestimate}. The following inequalities are immediate from \eqref{eq:Pell_Type}, and  will be useful in the sequel: 

\begin{equation}\label{280120.1}
 \dfrac{|t|^2-4}{|D|}\leq |u|^2 \leq \dfrac{|t|^2+4}{|D|},
\end{equation}
Moreover, if $u \neq 0,$ then 
\begin{equation}\label{280120.3}
 |t|^2 \ge |D|-4.
\end{equation}

\begin{lemma}\label{lemmaforsystole}
Let $D\in\mathfrak{D}$, and $\epsilon_D=\frac{1}{2}(t_0+u_0\sqrt{D})$ be a fundamental solution of $t^2-u^2D=4$. Then, if $|D|\geq 52$ and $\epsilon_D^{n+1}=\frac{1}{2}(t_n+u_n\sqrt{D})$ we have $|t_n|<\frac{4}{9}|u_n|^2$ for $n\geq 2$.
\end{lemma}

\begin{proof}
By the left-hand side of \eqref{280120.1}, we have $|u_n|^2 \ge \frac{|t_n|^2-4}{|D|}$. Hence, it is sufficient to prove that for $n \geq 2$
\begin{equation}\label{quadraticfunction}
 \frac{|t_n|^2-4}{|D|}>\frac{9}{4}|t_n|.
\end{equation}

By \eqref{280120.3}, we have $|t_0|^2 \ge |D|-4$ since $u_0 \neq 0.$ Moreover, by Proposition \ref{Pliniobreak} for $n=0$ we obtain 
\begin{equation*}
|\varepsilon_D|^2 \ge |D|-7 > \frac{1}{2}|D|.
\end{equation*}

Together with Proposition \ref{Pliniobreak} we get for $n \ge 2$ that  
\begin{align} \label{growthoftn} 
|t_n|^2 & >  \left(\frac{1}{2}|D| \right)^{(n+1)}-3 > 25 \left(\frac{1}{2}|D| \right)^2,
\end{align}
where the last inequality holds for $|D|\geq 52$ since the function $x^3-25x^2-3$ is positive for $x\geq 26$. Now we are able to ensure \eqref{quadraticfunction} for $n \ge 2$. Indeed, by \eqref{growthoftn} $$|t_n| > \frac{5}{2}|D|>1,$$  for any $n \ge 2$. Hence, 
$$|t_n|^2-4>\frac{5}{2}|D||t_n|-4 > \frac{9}{4}|D||t_n|.$$\qedhere         
\end{proof}
In order to apply Proposition \ref{systolestimate}, it is natural to define the number

 $$m(\epsilon_D)=\min \left\lbrace n\ge 0: |t_n|<\frac{4}{9}|u_n|^2 \right\rbrace.$$ 

By the previous lemma $m(\epsilon_D)\in\{0, 1, 2\}$ whenever $|D|\geq 52$. The following lemma proves that we can control the size of $u_{m(\epsilon_D)}$ by $|D|$ when $m(\epsilon_D)\neq 0$.  

\begin{lemma}\label{coffeelemma}
Let $D\in\mathfrak{D}$ with $|D|\geq 52$. If $m(\epsilon_D)\neq 0,$ then $|u_{m(\epsilon_D)}|\leq 30|D|^{3/2}$.
\end{lemma}

\begin{proof}
 By definition, $t_n$ and $u_n$ satisfy $\epsilon_D^{n+1}=\frac{1}{2}(t_n+u_n\sqrt{D})$. Hence, we can write 
\begin{equation}\label{recurrence}
 u_1=u_0t_0 \hspace{0.5cm} \mbox{ and } \hspace{0.5cm} u_2=\frac{1}{2}(u_0t_1+u_1t_0).
\end{equation}

Note that if $t^2-Du^2=4$ with $u \neq 0$ and $|t| \ge \frac{4}{9}|u|^2,$ then
\begin{equation} \label{upperboundt}
 |t| \leq \frac{9}{2} |D|.
\end{equation}

Indeed, by the argument used in Lemma \ref{lemmaforsystole}, if $t, u$ satisfy $|t| \ge \frac{4}{9}|u|^2$ then \eqref{quadraticfunction} does not hold for $t_n=t$. By the left-hand side of \eqref{280120.1} we have
$$|t| \geq \frac{4}{9}|u|^2 \ge \frac{4}{9}\left( \frac{|t|^2-4}{|D|}\right) \Leftrightarrow |t|^2-\frac{9}{4}|D||t|-4 \le 0. $$

Therefore, $$|t| \le \frac{1}{2}\left(\frac{9}{4}|D|+\sqrt{\left(\frac{9}{4}|D|\right)^2+16} \right).$$

Since we can assume that $16 \le 3\left(\frac{9}{4}|D|\right)^2$, we obtain \eqref{upperboundt}.
Suppose now that $m(\varepsilon_D)=1$. By hypothesis, $|t_0|^2 \ge |D|-4 \ge 4,$ and we can apply the right-hand side of \eqref{280120.1} for $u_0,$ \eqref{upperboundt} for $t_0$  and \eqref{recurrence}  to obtain that $$|u_1|^2=|u_0|^2|t_0|^2 \leq \frac{(|t_0|^2+4)}{|D|}|t_0|^2 \leq \frac{2|t_0|^4}{|D|} \le  2\left( \frac{9}{2} \right)^4|D|^3. $$

Hence,
$$|u_1| \le 30|D|^\frac{3}{2} \hspace{0.3cm} \mbox{ if } \hspace{0.3cm} m(\varepsilon_D)=1.$$

Analogously, if $m(\varepsilon_D)=2,$ together with $|t_i|^2\geq 4$ for $i=0,1,$ we can apply the right-hand side of \eqref{280120.1} for $u_0$ and $u_1,$ \eqref{upperboundt} for $t_0$ and $t_1$.  By \eqref{recurrence} we have,

$$|u_2| \leq \frac{1}{2}\left(\frac{\sqrt{2}|t_0||t_1|}{\sqrt{|D|}}+\frac{\sqrt{2}|t_1||t_0|}{\sqrt{|D|}} \right) \leq \sqrt{2}\left(\frac{9}{2}\right)^2|D|^\frac{3}{2}$$

Therefore,
$$|u_2| \le 30|D|^\frac{3}{2} \hspace{0.3cm} \mbox{ if } \hspace{0.3cm} m(\varepsilon_D)=2.$$\end{proof}

\section{Proof of the main theorem} \label{sec:proofmainthm}

We can now prove the main result of this article (see Theorem \ref{thm:main_theorem}). Before that, let us explain a notation that will be used through this section. We will say that two positive functions $f,g$ satisfy the relation $f(x)\gtrsim g(x)$ if for any $\epsilon>0$ there exists $x_0=x_0(\epsilon)$ such that $f(x)\geq (1-\epsilon)g(x)$ for $x>x_0$. 

 We will start by defining the manifolds $M_i$. Let $d>0$ be a square free rational integer such that $\kd$ has class number one. Recall that $h(D)$ is defined as the number of inequivalent binary quadratic forms over $\Od$ with discriminant equal to $D$ (see Section \ref{sec:quadforms}).  By Theorem \ref{thm:Sarnak_growth} there exists a sequence $D_i\in\mathfrak{D}$ with $|\epsilon_{D_{i}}|\rightarrow\infty,$ and a constant $c_d>0$
such that 
\begin{equation}\label{eq:Sarnak_avarage}
h(D_i)  \gtrsim c_d \frac{|\epsilon_{D_i}|^2}{\log(|\epsilon_{D_i}|)}.
\end{equation}

Since $h(D_i)\rightarrow\infty$  we have that $|D_i|\rightarrow\infty,$ and we can assume that $D_i$ satisfies the condition of Lemma \ref{lemmaforsystole}. For any $i,$ consider a fundamental solution $$\epsilon_{D_{i}}=\frac{1}{2}(t_{0}(i)+u_{0}(i)\sqrt{D_i})$$ of the Pell's type equation $x^2-D_iy^2=4$ over $\Od\times\Od$. We define the numbers
\begin{align*}
m_i&:=m(\epsilon_{D_i}), \hspace{0.5cm} u_i:=u_{m_i}, \hspace{0.5cm} t_i:=t_{m_i},
\hspace{0.5cm}
\end{align*}
as in Section \ref{sec:cong_subpgs_and_systoles}. From $u_i, t_i, D_i$ we have constructed a class $\tau_i\in\Od/u_i\Od$ (see Remark \ref{rem:tau}), and the corresponding group $$\Gamma_i=\Bd_{\tau_i}[u_i].$$

 By Proposition \ref{Pliniobreak}, \eqref{incongruence2}, and \eqref{030220.5}, whenever $|\epsilon_{D_i}|$ is large enough, any non parabolic element in $\Gamma_i$ is loxodromic. This implies that the group $\Gamma_i$ is torsion free, and $M_i=\Gamma_i\backslash\Hyp$ is a  hyperbolic $3$-manifold of finite volume.

\begin{proposition}\label{prop:many_systoles_1}
The sequence of hyperbolic 3-manifolds $\{M_i\}$ satisfies $$\K(M_i)\geq h(D_i).$$
\end{proposition}

\begin{proof}
By Lemma \ref{lemmaforsystole} the pair $t_i, u_i$ satisfies the hypothesis of  Proposition \ref{systolestimate} for any $i$. Hence, the elements in $\Gamma_i$ with trace equal to $t_i$ induce closed geodesics of length equal to the systole of $M_i$. It is then enough to prove that there exist at least $h(D_i)$ pairwise nonconjugated loxodromic elements in $\Gamma_i$, with trace equal to $t_i$


By definition, there exist $h(D_{i})$ no pairwise equivalent primitive quadratic forms $Q_1, Q_2,\ldots,Q_{h(D_{i})}$ of discriminant $D_i$. By Proposition \ref{prop:forms_and_loxodromic}, these quadratic forms correspond to $\epsilon_{Q_1}=A_1,\ldots, \epsilon_{Q_{h(D_{i})}}=A_{h_{(D_{i})}}$ no pairwise conjugated primitive matrices in $\Bd$ with $\tr(A_j)=t_0(i),$ for any $j=1,\ldots,h(D_i)$ (see \eqref{eq:matrixepsilon}). The explicit correspondence implies that $A_j^{m_i+1}\in\Gamma_i$. Moreover $\tr(A_j^{m_i+1})=t_i$ and then   $A_j^{m_i+1}$ induces a closed geodesic $\gamma_j$ of $M_i$ of shortest length for any $j$. We claim that the powers $A_j^{m_i+1}$ are pairwise nonconjugated elements in $\Gamma_i$. In fact, the matrices $A_j$ satisfy the characteristic equation $$X^2-t_0(i)X+1=0,$$
and then $A_j^{m_i+1}$ can be written as a linear polynomial on $A_j$, whose coefficients depend only on $t_0(i)$. This implies that $A_j^{m_i+1}$ would be conjugated in $\Gamma_i$ to $A_k^{m_i+1}$ if and only if $A_j$ and $A_k$ were conjugated in $\Gamma_i$. \qedhere
\end{proof}

We can go a step forward to produce more closed geodesics of shortest length. Since $\Gamma_{i}$ is a normal subgroup of $\SL(\Od)$, the group $G_i=\SL(\Od)/\Gamma_{i}$ is a subgroup of isometries of $M_i$. The main idea now is to consider the action of $G_{i}$ on the set of closed geodesics of $M_i$.

\begin{lemma}\label{lem:orbits}
Let $\gamma_1,\ldots,\gamma_{h(D_i)}$ the closed geodesics of shortest length in $M_i$ obtained in Proposition \ref{prop:many_systoles_1}. Then the orbits $G_i\cdot\gamma_j,$ $j=1,\ldots,h(D_i)$ are disjoint.

\end{lemma} 

\begin{proof}
Indeed, the action of $G_i$ in the set of closed geodesics translates to the action of $\Bd$ by conjugation on the conjugacy classes of loxodromic elements of $\Gamma_i$. Hence, if the closed geodesic $\gamma_j$ lies in the same orbit of $\gamma_k,$ then $A_j^{m_i+1}=C(A_k^{^{m_i+1}})C^{-1}$ for some $C \in \Bd$. However, as we have seen in the proof of Proposition \ref{prop:many_systoles_1}, the matrices $A_j^{m_i+1}$ and $A_k^{m_i+1}$ are pairwise nonconjugated elements in $\Bd$.
\end{proof}

Note that, since $|\epsilon_{D_{i}}|\rightarrow\infty,$ then $|u_i|\rightarrow\infty$  (see Proposition \ref{Pliniobreak} and Lemma \ref{lemmaforsystole}). We will need to compare the volume growth of $M_i$ with the growth of $|u_i|$. To do so, it will be more convenient to use the notation $\N(u)=u \overline{u}=|u|^2$ as the usual \emph{norm} of an imaginary quadratic number, in order to refer to some known results.  

\begin{lemma}\label{030220.1}
 The sequences $\N(u_i)$ and $\vol(M_i)$ are related by
 $$\N(u_i) \ge \mu ~  \vol(M_i)^{\frac{1}{3}},$$
 where $\mu>0$ does not depend on $i$.
\end{lemma}
\begin{proof}
 The manifold $M_i$ is a normal covering of $\Bd\backslash\mathbb{H}^3$ with fibers of cardinality $[\SL_2(\Od):\Gamma_i]$. 
 Since $\Bd[u_i \Od]$ is a normal subgroup of $\Gamma_i$ of index 2, we have that 
 \begin{eqnarray} \label{030220.2}
  [\Bd:\Gamma_i]=2\cdot[\Bd:\Bd[u_i\Od]].
 \end{eqnarray}

On the other hand, it is well known that 
\begin{align} \label{030220.3}
 C ~ \N(u_i)^3 \leq [\Bd:\Bd[u_i\Od]] \leq  \N(u_i)^{3},
\end{align}
for some constant $C>0$ depending only on $k_d$ (\textit{c.f} \cite[Cor. 4.6]{KSV07}, \cite[Lem. 4.1]{Mur17}, see also \cite[Sec. 5]{Kuch15}).  
 Since  $\vol(M_i)=\vol(\Bd\backslash\mathbb{H}^3)[\SL_2(\Od):\Gamma_i]$ we obtain the result.
\end{proof}

\begin{proof}[Proof of Theorem \ref{thm:main_theorem}]

It follows from Proposition \ref{prop:many_systoles_1} and Lemma \ref{lem:orbits} that any $M_i$ satisfies 
\begin{equation}\label{eq:kiss_stabilizer}
\K(M_i)\ge\displaystyle\sum_{j=1}^{h(D_i)}|G_i\cdot\gamma_j|.
\end{equation}
On the other hand, the cardinality of the isotropy group $(G_i)_{\gamma_j}$ of $\gamma_j$ is equal to $[C(A_j^{m_i+1}):\langle A_j^{m_i+1}\rangle],$ where $C(A_j^{m_i+1})$ denotes the centralizer of $A_j^{m_i+1}$. Since each $A_j$ is primitive, it is well known that $C(A_{j}^{m_i+1})=\langle A_j, -\mathrm{Id}\rangle$. Hence $|(G_i)_{\gamma_j}| = 2(m_i +1) \leq 6,$ and  \eqref{eq:kiss_stabilizer} implies

\begin{align}
\K(M_i) \geq \sum_{j=1}^{h(D_i)}\frac{|G_i|}{|(G_i)_{\gamma_{j}}|}
        \geq\frac{h(D_i)|G_i|}{6}\label{eq:kiss_class_number_isometry_group}.
\end{align}

We now need to compare the quantities $|G_i|$ and $h(D_i)$ to the volume of $M_i$. To do that, we will relate these quantities to $\N(u_i),$  in view of Lemma \ref{030220.1}

Since $|G_i|=[\Bd:\Gamma_i]$, by \eqref{030220.2} and the left-hand side of \eqref{030220.3} we have 
 $$|G_i|\ge C\N(u_i)^{3}.$$

The next step is to bound $h(D_i)$ in terms of $\N(u_i)$. By the relation $u_i^2D_i=t_i^2-4$ we have $|t_i|^2 \gtrsim \N(u_i)|D_i|,$ and it follows from Proposition \ref{Pliniobreak} that $|t_i|^2 \sim |\epsilon_{D_i}|^{2(m_i+1)}.$ Hence
\begin{align}\label{030220.4}
 |\epsilon_{D_i}|^2 \gtrsim |t_i|^{\frac{2}{m_i+1}} \gtrsim \left( \N(u_i)|D_i| \right)^{\frac{1}{m_i+1}}.
\end{align}

The function $x \mapsto \frac{x}{\log(x)}$ is increasing for large $x,$ and we apply this to \eqref{eq:Sarnak_avarage} and \eqref{030220.4} to obtain  
\begin{equation}\label{exponents_cases}
h(D_i) \gtrsim 2c_d \frac{(\N(u_i)|D_i|)^{\frac{1}{m_i+1}}}{\log((\N(u_i)|D_i|)^{\frac{1}{m_{i}+1}})}.
\end{equation}

 Since $m_i \in \{0,1,2\}$ we can suppose that the sequence $m_i$ is constant and equal to $m$. We have two cases to consider.

\textit{Case $m=0$}: In this case we use the simple lower bound $|D_i| >1$ to see that $|\N(u_i)D_i|>|\N(u_i)|$. Now, using again the fact that $x\mapsto\frac{x}{\log x}$ is increasing for large $x,$ we apply \eqref{exponents_cases}  in order to get 
\begin{equation}\label{eq:hd_1}
h(D_i) \gtrsim 2c_d \frac{\N(u_i)}{\log(\N(u_i))}.
\end{equation}

\textit{Case $m \in \{1,2\}$}: In this case we use Lemma \ref{coffeelemma} to obtain \(|D_i| \ge 30^{-\frac{2}{3}} \N(u_i)^{\frac{1}{3}}\)
and then
\begin{equation} \label{blackboardfriday}
\left(\N(u_i)|D_i|\right)^{\frac{1}{m+1}} \ge 30^{-\frac{2}{3(m+1)}} \N(u_i)^{\frac{4}{3(m+1)}}.
\end{equation}
Using once more that $x\mapsto\frac{x}{\log x}$ is increasing for large $x,$ it follows from \eqref{exponents_cases} and \eqref{blackboardfriday} that
\begin{align*}
h(D_i) &\gtrsim 2c_d\cdot30^{-\frac{2}{3(m+1)}}\cdot \frac{\N(u_i)^{\frac{4}{3(m+1)}}}{\log\left(30^{-\frac{2}{3(m+1)}}\N(u_i)^{\frac{4}{3(m+1)}}\right)}\nonumber\\
&\gtrsim 2c_d\cdot30^{-\frac{2}{3(m+1)}}\cdot \frac{\N(u_i)^{\frac{4}{3(m+1)}}}{\log\left(\N(u_i)^{\frac{4}{3(m+1)}}\right)}.
\end{align*}
The last expression is minimized when $m=2$. Therefore, together with \eqref{eq:hd_1} we obtain that the sequence $\{D_i\}$ satisfies 
\begin{align*}
h(D_i)\gtrsim 2c_d\cdot30^{-\frac{2}{9}}\cdot \frac{\N(u_i)^{\frac{4}{9}}}{\log\left(\N(u_i)^{\frac{4}{9}}\right)}.
\end{align*}
Now, we can put together the lower bound for $|G_i|$ and $h(D_i)$ to obtain from \eqref{eq:kiss_class_number_isometry_group} that 
\begin{align*}
\K(M_i)\gtrsim \frac{c_d\cdot C\cdot30^{-\frac{2}{9}}}{3}\cdot \frac{\N(u_i)^{\frac{31}{9}}}{\log\left(\N(u_i)^{\frac{4}{9}}\right)}.
\end{align*}

The result now follows from Lemma \ref{030220.1} and the monotonicity of $\frac{x}{\log x}$ for large $x$. \qedhere

\end{proof}

\begin{corollary}[of the proof]\label{cor:final}
The manifolds $\{M_i\}$ given in the proof of Theorem \ref{thm:main_theorem} satisfy
 \[ \sys(M_i) \gtrsim \frac{1}{3}\log(\vol(M_i)). \] 
\end{corollary}

\begin{proof}
 By Proposition \ref{systolestimate} it holds the equality \[ 4\cosh(\sys(M_i))=|t_i^2|+|t_i^2-4|=|t_i|^2+|D_i||u_i|^2. \]

Since $|D_i|>1$, then
 \[\sys(M_i) \geq \cosh^{-1}\left(\frac{\N(u_{i})}{4}\right) \geq \log(\N(u_i))-\log(4).\]
 
By Lemma \ref{030220.1}, we have \[\sys(M_i) \geq \frac{1}{3}\log(\vol(M_i))+\log(\mu)-\log(4). \]
Since $\vol(M_i) \to \infty$ and $\mu$ does not depend on $i$, we get
 \[ \sys(M_i) \gtrsim \frac{1}{3}\log(\vol(M_i)). \]  
\end{proof}

\bibliographystyle{plain}

 \vspace{4mm}
  Cayo D\'oria\\
Universidade de Goi\'as, Instituto de Matem\'atica e Estat\'istica.\\ Rua Jacarand\'a -- Ch\'acaras Calif\'ornia, 74001-970. Goi\^ania - GO, Brazil.\\
 \textit{E-mail address}:\hspace{2mm}\texttt{cayodoria@ufg.br}

 \vspace{6mm}
 \hspace{-4mm}Plinio G. P. Murillo\\
Universidade Federal Fluminense, Instituto de Matem\'atica e Estat\'istica.\\
Rua Prof. Marcos Waldemar de Freitas Reis, S/n, Bloco H, Campus do Gragoat\~a, 24210-201. Niter\'oi - RJ, Brazil.\\
 \textit{E-mail address}:\hspace{2mm}\texttt{pliniom@id.uff.br}

\end{document}